\documentclass[10pt]{amsart}

\usepackage[colorlinks]{hyperref}
\usepackage{color,graphicx,shortvrb}
\usepackage[latin 1]{inputenc}

\usepackage[active]{srcltx} 

\usepackage{enumerate}

\newtheorem{theorem}{Theorem}[section]

\newtheorem{corollary}[theorem]{Corollary}

\newtheorem{lemma}[theorem]{Lemma}

\newtheorem{proposition}[theorem]{Proposition}
\newtheorem{remark}[theorem]{Remark}

\newtheorem{example}[theorem]{Example}

\def\RR{{\mathbb{R}}}

\def\11{\textbf{$1$}}
\def\CC{{\mathbb{C}}}

\DeclareGraphicsExtensions{.jpg,.pdf,.png,.eps}

\begin{document}

\title[Orthogonal forms and orthogonality preservers]{Orthogonal forms and orthogonality preservers on real function algebras}

\author[J.J. Garc{\' e}s]{Jorge J. Garc{\' e}s}
\email{jgarces@ugr.es}
\address{Departamento de An{\'a}lisis Matem{\'a}tico, Facultad de
Ciencias, Universidad de Granada, 18071 Granada, Spain.}

\author[A.M. Peralta]{Antonio M. Peralta}
\email{aperalta@ugr.es}
\address{Departamento de An{\'a}lisis Matem{\'a}tico, Facultad de
Ciencias, Universidad de Granada, 18071 Granada, Spain.}

\thanks{Authors partially supported by the Spanish Ministry of Economy and Competitiveness,
D.G.I. project no. MTM2011-23843, and Junta de Andaluc\'{\i}a grants FQM0199 and
FQM3737.}



\date{}

\begin{abstract} We initiate the study of orthogonal forms on a real C$^*$-algebra.
Motivated by previous contributions, due to Ylinen, Jajte, Paszkiewicz and Goldstein,
we prove that for every continuous orthogonal form $V$ on a commutative real C$^*$-algebra, $A$,
there exist functionals $\varphi_1$ and $\varphi_2$ in $A^{*}$ satisfying
$$V(x,y) = \varphi_1 (x y) + \varphi_2 ( x y^*),$$ for every $x,y$ in $A$.  We describe the general form of a
(not-necessarily continuous) orthogonality preserving linear map
between unital commutative real C$^*$-algebras. As a consequence,
we show that every orthogonality preserving linear bijection
between unital commutative real C$^*$-algebras is continuous.
\end{abstract}

\maketitle
\thispagestyle{empty}

\section{Introduction and preliminaries}

Elements $a$ and $b$ in a real or complex C$^*$-algebra, $A$,
are said to be \emph{orthogonal} (denoted by $a \perp
b$) if $a b^* = b^* a=0$. A bounded bilinear form $V: A\times A \to \mathbb{K}$
is called \emph{orthogonal} (resp., \emph{orthogonal on self-adjoint elements})
whenever $V(a,b^*)=0$ for every $a\perp b$ in $A$ (resp., in the self-adjoint part of $A$). All the forms considered in this paper are assumed to be continuous.
Motivated by the seminal contributions by K Ylinen \cite{Yl75} and R. Jajte and A. Paszkiewicz \cite{JajPaszk}, S. Goldstein proved that every orthogonal form $V$ on a (complex) C$^*$-algebra, $A,$ is of the form $$V(x,y) = \phi (x y ) +  \psi (x y ) \ \ (x,y\in A),$$ where $\phi$ and $\psi$ are two functionals in $A^*$ (cf. \cite[Theorem 1.10]{Gold}). A simplified proof of Goldstein's theorem was published by U. Haagerup and N.J. Laustsen in \cite{HaaLaut}. This characterisation has emerged as a very useful tool in the study of bounded linear operators between C$^*$-algebras which are orthogonality or disjointness preserving (see, for example, \cite{BurFerGarMarPe,BurFerGarPe}).

The first aim of this paper is to study orthogonal forms on the wider class of real C$^*$-algebras. Little or nothing is known about the structure of an orthogonal form $V$ on a real C$^*$-algebra. At first look, one is tempted to consider the canonical complex bilinear extension of $V$ to a form on the complexification, $A_{\mathbb{C}} = A\oplus iA$, of $A$ and, when the latter is orthogonal, to apply Goldstein's theorem. However, the complex bilinear extension of $V$ to $A_{\mathbb{C}} \times A_{\mathbb{C}}$, need not be, in general, orthogonal (see Example \ref{exam complexification}). The study of orthogonal forms on real C$^*$-algebras requires a completely independent strategy; surprisingly the resulting forms will enjoy a different structure to that established by S. Goldstein in the complex setting.

In section 2 we establish some structure results for orthogonal forms on a general real C$^*$-algebra, showing, among other properties, that every orthogonal form on a real C$^*$-algebra extends to an orthogonal form on its multiplier algebra (see Proposition \ref{prop extns to the multiplier}). It is also proved that, for each orthogonal and symmetric form $V$ on a real C$^*$-algebra, $A,$ there exists a functional $\phi\in A^*$ satisfying $V(a,b) =\phi (a b+ ba),$ for every $a,b\in A$ with $a=a^*$, $b^*=b$ (cf. Proposition \ref{p orthog forms on Asa}). In the real setting, the skew-symmetric part of a real C$^*$-algebra, $A,$ is not determined by the self-adjoint part of $A,$ so the information about the behavior of $V$ on the rest of $A$ is very limited.

Section 3 contains one of the main results of the paper: the characterisation of all orthogonal forms on a commutative real C$^*$-algebra. Concretely, we prove that a form $V$ on a commutative real C$^*$-algebra $A$ is orthogonal if, and only if, there exist functionals $\varphi_1$ and $\varphi_2$ in $A^{*}$ satisfying $$V(x,y) = \varphi_1 (x y) + \varphi_2 ( x y^*),$$ for every $x,y\in A$ (see Theorem \ref{t orhtogonal forms real abelian}). Among the consequences, it follows that the complex bilinear extension of $V$ to the complexification of $A$ is orthogonal if, and only if, we can take $\varphi_2=0$  in the above representation.

We recall that a mapping $T: A\to B$ between real or complex C$^*$-algebras is said to be \emph{orthogonality or disjointness preserving} (also called \emph{separating}) whenever $a\perp b$ in $A$ implies $T(a)\perp T(b)$ in $B.$  The mapping $T$ is \emph{bi-orthogonality preserving} whenever the equivalence $$ a\perp b \Leftrightarrow T(a) \perp T(b)$$ holds for all $a,b$ in $A$. As noticed in \cite{BurGarPe}, every bi-orthogonality preserving linear
surjection, $T:A\to B$ between two C$^*$-algebras is injective.

The study of orthogonality preserving operators between
C$^*$-algebras started with the work of W. Arendt \cite{Arendt} in the
setting of unital abelian C$^*$-algebras. Subsequent contributions by K. Jarosz \cite{Jar}
extended the study to the setting of orthogonality preserving (not
necessarily bounded) linear mappings between abelian
C$^*$-algebras. The first study on orthogonality preserving symmetric (bounded) linear operators between general (complex)
C$^*$-algebras is originally due to M. Wolff (cf. \cite{Wol}). Orthogonality preserving
bounded linear maps between C$^*$-algebras, JB$^*$-algebras and JB$^*$-triples were completely described in
\cite{BurFerGarMarPe} and \cite{BurFerGarPe}.

The pioneer works of E. Beckenstein, L. Narici, and A.R. Todd in
\cite{BeckNarTodd} and \cite{BeckNarTodd88} (see also
\cite{BeckNar}) were applied by K. Jarosz to prove that every
orthogonality preserving linear bijection between $C(K)$-spaces is
(automatically) continuous (see \cite{Jar}). More recently, M.
Burgos and the authors of this note proved in \cite{BurGarPe} that
every bi-orthogonality preserving linear surjection between two
von Neumann algebras (or between two compact C$^*$-algebras) is
automatically continuous (compare \cite{OikPeRa}, \cite{OikPePu}
for recent additional generalisations).

The main goal of section 4 is to describe the orthogonality preserving linear mappings between
unital commutative real C$^*$-algebras
(see Theorem \ref{t op between Cr(K)}). As a consequence, we shall prove that
 every orthogonality preserving linear bijection between unital commutative real
 C$^*$-algebras\hyphenation{al-gebras} is automatically continuous. We shall exhibit
 some examples illustrating that the results in the real
setting are completely independent from those established for
complex C$^*$-algebras. We further give a characterisation of those linear mappings between real forms of $C(K)$-spaces which are bi-orthogonality preserving.

\subsection{Preliminary results}

Let us now introduce some basic facts and definitions required later.
A \emph{real C$^*$-algebra} is a real Banach *-algebra $A$ which satisfies
the standard C$^*$-identity, $\|a^*a\|=\|a\|^2$, and which also has the property
that $1 + a^{\ast}a$ is invertible in the unitization of $A$ for every $a\in A$.
It is known that a real Banach *-algebra, $A,$ is a real C$^*$-algebra if, and only if, it is
isometrically *-isomorphic to a norm-closed real *-subalgebra of bounded
operators on a real Hilbert space (cf. \cite[Corollary 5.2.11]{Li}).

Clearly, every (complex) C$^*$-algebra is a real C$^*$-algebra when scalar multiplication is
restricted to the real field. If $A$ is a real C$^*$-algebra whose algebraic complexification
is denoted by ${B}= A\oplus i  A,$ then there exists a C$^*$-norm on ${B}$
extending the norm of $A$. It is further known that there exists an involutive conjugate-linear
$^*$-automorphism $\tau$ on ${B}$ such that $A = {B}^{\tau} :=\{ x\in B : \tau (x) =x\}$
(compare \cite[Proposition 5.1.3]{Li} or \cite[Lemma 4.1.13]{Rick}, and \cite[Corollary 15.4]{Goo}).
The dual space of a real or complex C$^*$-algebra $A$ will be denoted by $A^*$.
Let $\widetilde{\tau}:B^*\rightarrow B^*$ denote the map defined by
\begin{equation*}
\widetilde{\tau}(\phi)(b)=\overline{\phi(\tau(b))}\qquad(\phi \in B^*,\,b\in B).
\end{equation*} Then $\widetilde{\tau}$ is a conjugate-linear isometry of period 2 and the mapping
$$(B^*)^{\widetilde{\tau}} \to A^*$$ $$\varphi \mapsto \varphi|_{A}$$ is a surjective linear isometry.
We shall identify $(B^*)^{\widetilde{\tau}}$ and $A^*$ without making any explicit mention.

When $A$ is a real or complex C$^*$-algebra, then $A_{sa}$ and $A_{skew}$ will
stand for the set of all self-adjoint and skew-symmetric elements in $A$, respectively.
We shall make use of standard notation in C$^*$-algebra theory.

Given Banach spaces $X$ and $Y$, $L(X,Y)$ will denote the space of
all bounded linear mappings from $X$ to $Y$. We shall write $L(X)$
for the space $L(X,X)$. Throughout the paper the word
``operator'' (respectively, multilinear or sesquilinear operator) will always
mean bounded linear mapping (respectively bounded multilinear or sesquilinear mapping).
The dual space of a Banach space $X$ is always denoted by $X^*$.

Let us recall that a series $\sum_n x_n$ in a Banach space is called \emph{weakly unconditionally
Cauchy (w.u.C.)} if there exists $C>0$ such that for any finite subset
$F\subset{\mathbb{N}}$ and $\varepsilon_{n} = \pm 1$ we have $\displaystyle\left\| \sum_{n\in F}
\varepsilon_{n} x_{n} \right\| \leq C$. A (linear) operator
$T:X\longrightarrow Y$ is \emph{unconditionally converging} if for every w.u.C.
series $\sum_{n} x_{n}$ in $X,$ the series $\sum_n T(x_n)$ is unconditionally convergent in $Y$, that is,
every subseries of $\sum_{n} T(x_{n})$ is norm converging. It is known that $T:X\to Y$ is unconditionally converging if, and only if, for every w.u.C. series $\sum_{n} x_{n}$ in $X,$ we have $\|T(x_n)\| \to 0$ (compare, for example, \cite[page 1257]{PeViWriYli2})

Let us also recall that a Banach space $X$ is said to have \emph{Pe\l czy\'nski's}
\emph{property} (V) if, for every Banach space $Y$,
every unconditionally converging operator $T: X \to Y$ is weakly compact.

The proof of the following elementary lemma is left to the reader.

\begin{lemma}\label{l (V) real forms} Let $X$ be a complex Banach space, $\tau: X\to X$ a conjugate-linear
period-2 isometry. Then the real Banach space $X^{\tau}:=\{x\in X : \tau(x) = x\}$ satisfies property (V) whenever $X$ does. $\hfill\Box$
\end{lemma}

We shall require, for later use, some results on extensions of multilinear operators.
Let $X_1,\dots,X_n$, and $X$ be Banach spaces,
$T:X_1\times\cdots\times X_n\to X$ a (continuous) $n$-linear
operator, and $\pi:\{1,\dots,n\}\to \{1,\dots,n\}$ a permutation.
It is known that there exists a unique $n$-linear extension
$AB(T)_\pi:X_1^{**}\times\cdots\times X_n^{**}\to X^{**}$ such
that for every $z_i\in X_i^{**}$ and every net
$(x_{\alpha_i}^i)\in X_i$ ($1\leq i \leq n$), converging to $z_i$
in the weak* topology we have
$$AB(T)_\pi(z_1,\ldots,
z_n)=\mbox{weak*-}\lim_{\alpha_{\pi(1)}}\cdots
\mbox{weak*-}\lim_{\alpha_{\pi(n)}} T(x_{\alpha_1}^1,\ldots,
x_{\alpha_n}^n).$$

Moreover, $AB(T)_\pi$ is bounded and has the same norm as $T$. The
extensions $AB(T)_\pi$ coincide with those considered by Arens in \cite{Are1,Are2}
and by Aron and Berner for polynomials in \cite{ArBer}. The $n$-linear operators $AB(T)_\pi$ are
usually called the {\em Arens} or \emph{Aron-Berner extensions} of $T$.

Under some additional hypothesis, the Arens extension of a multilinear operator also is separately weak$^*$ continuous. Indeed, if every operator from $X_i$ to $X_j^*$ is weakly compact ($i\not
= j)$ the Arens extensions of $T$ defined above do not
depend on the chosen permutation $\pi$ and they are all separately weak$^*$ continuous
(see \cite{ArCoGa}, and Theorem 1 in \cite{BomVi}). In particular, the above requirements always hold
when every $X_i$ satisfies Pelczynski's property $(V)$ (in such case $X_i^*$ contains no copies
of $c_0$, therefore every operator from $X_i$ to $X_j^*$ is
unconditionally converging, and hence weakly compact by property $(V)$, see
\cite{Pel}). When all the Arens extensions of $T$ coincide,
the symbol $AB(T)=T^{**}$ will denote any of them.

We should note at this
point that every C$^*$-algebra satisfies property $(V)$ (cf.
Corollary 6 in \cite{Pfi}). Since every real C$^*$-algebra is, in particular,
a real form of a (complex) C$^*$-algebra, it follows from Lemma \ref{l (V) real forms}
that every real C$^*$-algebra satisifes property $(V)$. We therefore have:

\begin{lemma}\label{l extensions real forms} Let $A_1,\ldots, A_k$ be real C$^*$-algebras and let $T$
be a multilinear continuous operator from $A_1\times \ldots \times A_k$ to a real Banach space $X$. Then $T$ admits a
unique Arens extension $T^{**}:A_1^{**}\times \ldots \times A_k^{**}\to X^{**}$ which is separately weak$^*$ continuous.$\hfill\Box$
\end{lemma}

Given a real or complex C$^*$-algebra, $A$, the
\emph{multiplier algebra} of $A$, $M(A)$, is the set of all
elements $x\in A^{**}$ such that, for each element $a\in A$, $x
a$ and $a x$ both lie in $A$. We notice that $M(A)$ is a C$^*$-algebra
and contains the unit element of $A^{**}$. It should be recalled here that
$A= M(A)$ whenever $A$ is unital.

\begin{proposition}\label{prop extns to the multiplier} Let $A$ be a real C$^*$-algebra.
Suppose that $V: A\times A\to \mathbb{R}$ is an orthogonal bounded bilinear form.
Then the continuous bilinear form $$\tilde{V} : M(A)\times M(A) \to \mathbb{R}, \ \ \tilde{V} (a,b)
:= V^{**} (a,b)$$ is orthogonal.
\end{proposition}

\begin{proof} Let $a$ and $b$ be two orthogonal elements in $M(A)$. Let $a^{[\frac13]}$ (resp.,
$b^{[\frac13]}$) denote the unique element $z$ in $M(A)$ satisfying $z z^* z =a$ (resp., $z z^* z =b$).
We notice that $a^{[\frac13]}$ and $b^{[\frac13]}$ are orthogonal, so, for each pair $x,y$
in $A$, $a^{[\frac13]} x a^{[\frac13]}$ and $b^{[\frac13]} y
b^{[\frac13]}$ are orthogonal elements in $A$. Since $V$ is orthogonal, we have
$$V(a^{[\frac13]} x a^{[\frac13]}, (b^{[\frac13]})^* y (b^{[\frac13]})^*) =0$$ for every $x,y \in A$.\smallskip

Goldstine's theorem (cf. Theorem V.4.2.5 in \cite{DunSchw})
guarantees that the closed unit ball of $A$ is weak*-dense in
the closed unit ball of $A^{**}$. Therefore we can pick two
bounded nets $(x_{\lambda})$ and $(y_{\mu})$ in $A$,
converging in the weak$^*$ topology of $A^{**}$ to $(a^{[\frac13]})^*$ and
$b^{[\frac13]}$, respectively.\smallskip

We have already mentioned that $V^{**}: A^{**}\times A^{**} \to \mathbb{R}$
is separately weak$^*$ continuous. Since $0= V(a^{[\frac13]} x_{\lambda} a^{[\frac13]}, (b^{[\frac13]})^* y_{\mu} (b^{[\frac13]})^*),$
for every $\lambda$ and $\mu$, taking limits, first in $\lambda$ and subsequently in $\mu$, we deduce that
$$V^{**}(a^{[\frac13]} (a^{[\frac13]})^* a^{[\frac13]}, (b^{[\frac13]})^* b^{[\frac13]} (b^{[\frac13]})^*) ) = \tilde{V}(a,b^*) =0,$$
which shows that $\tilde{V}$ is orthogonal.
\end{proof}

Since the multiplier algebra of a real or complex C$^*$-algebra always has a unit element,
Proposition \ref{prop extns to the multiplier} allows us to restrict our study on orthogonal
bilinear forms on a real C$^*$-algebra $A$ to the case in which $A$ is unital.\smallskip

A \emph{real von Neumann algebra} is a real C$^*$-algebra which is also a dual Banach space
(cf. \cite{IsRo} or \cite[\S 6.1]{Li}). Clearly, the self adjoint part of a real von Neumann algebra is a
JW-algebra in the terminology employed in \cite{HoS}, so every self-adjoint element in a real von Neumann algebra $W$
can be approximated in norm by a finite real linear combination of mutually orthogonal projections in $W$
(cf. \cite[Proposition 4.2.3]{HoS}). We shall explore now the validity in the real setting of
some of the results established by S. Goldstein in \cite{Gold}.

\begin{lemma}\label{l orthog forms on Asa}
Let $A$ be a real von Neumann algebra with unit $1$.
Suppose that $V: A\times A\to \mathbb{R}$ is a bounded bilinear form. The following are equivalent:
\begin{enumerate}[$(a)$] \item $V$ is orthogonal on $A_{sa}$;
\item $V(p,q) = 0$, whenever $p$ and $q$ are two orthogonal projections in $A$;
\item $V(a,b) = V(a b, 1)$ for every $a,b\in A_{sa}$ with $a b= b a$.
\end{enumerate}
If any of the above statements holds and $V$ is symmetric,
then defining $\phi_1 (x):=V(x,1)$ ($x\in A$),
we have $V(a,b) = \phi_1 (\frac{a b + ba}{2})$, for every $a,b\in A_{sa}$.
\end{lemma}

\begin{proof}
Applying the existence of spectral resolutions for self-adjoint elements in a real
von Neumann algebra, the argument given by S. Goldstein in \cite[Proposition 1.2]{Gold} remains valid
to prove the equivalence of $(a)$, $(b)$ and $(c)$.\smallskip

Suppose now that $V$ is symmetric. Let $a= \sum_{j=1}^{m} \lambda_j p_j$ be an algebraic element in $A_{sa}$,
where the $\lambda_j$'s belong to $\mathbb{R}$ and $p_1,\ldots, p_m$ are mutually orthogonal projections in $A$.
Since $V$ is orthogonal, for every projection $p\in A$, we have $$V(p,1) = V (p, 1-p) +V(p,p)= V(p,p).$$ Thus,
$$V(a,a) = \sum_{j=1}^{m} \lambda_j^2 V(p_j,p_j) = \sum_{j=1}^{m} \lambda_j^2 V(p_j,1) = V\left(  \sum_{j=1}^{m} \lambda_j^2 p_j,1\right) = V(a^2,1).$$
The (norm) density of algebraic elements in $A_{sa}$ and the continuity of $V$ imply that $V(a,a) = V(a^2,1),$ for every $a\in A_{sa}$.
Finally, applying that $V$ is symmetric we have $$V(a^2,1) + V(b^2,1) + V(ab + ba,1) = V((a+b)^2,1) $$
$$= V(a+b,a+b) = V(a,a)+V(b,b) + 2V(a,b),$$ for every $a,b\in A_{sa}$, and hence $V(a,b) = V(\frac{a b +ba}{2}, 1),$ for all $a,b\in A_{sa}$.
\end{proof}

The above result holds for every monotone $\sigma$-complete unital real C$^*$-algebra $A$
(that is, each upper bounded, monotone increasing
sequence of selfadjoint elements of $A$ has a least upper bound).\smallskip

Surprisingly, the final conclusion of the above Lemma can be established for unital
real C$^*$-algebras with independent basic techniques.

\begin{proposition}\label{p orthog forms on Asa}
Let $A$ be a unital real C$^*$-algebra with unit $1$.
Suppose that $V: A\times A\to \mathbb{R}$ is an orthogonal, symmetric, bounded, bilinear form.
Then defining $\phi_1 (x):=V(x,1)$ ($x\in A$), we have $V(a,b) = \phi_1 (\frac{a b+ ba}{2})$,
for every $a,b\in A_{sa}$.
\end{proposition}

\begin{proof} Let $a$ be a selfadjoint element in $A$. The real C$^*$-subalgebra, $C$, of $A$ generated by
$1$ and $a$ is isometrically isomorphic to the space $C(K,\mathbb{R})$ of all real-valued continuous
functions on a compact Hausdorff space $K$. The restriction of $V$ to $C\times C$ is orthogonal,
therefore the mapping  $x\mapsto V(x,x)$ is a $2$-homogeneous orthogonally additive polynomial on $C$.
The main result in \cite{PerezVi} implies the existence of a functional $\varphi_{a}\in C^*$ such that
$V(x,x) = \varphi_{a} (x^2)$, for every $x\in C$. It is clear that $\varphi_a (x) = V(x,1)$ for every $x\in C$.
In particular $$V(a,a) = \varphi_a (a^2) = V(a^2,1).$$ The argument given at the end of the proof of
Lemma \ref{l orthog forms on Asa} gives the desired statement.
\end{proof}

The above proposition shows that we can control the form of a symmetric orthogonal form on
the self adjoint part of a (unital) real C$^*$-algebra. The form on the skew-symmetric part remains out of control for the moment.

\section{Orthogonal forms on abelian real C$^*$-algebras}

Throughout this section, $A$ will denote a unital, abelian, real C$^*$-algebra whose
complexification will be denoted by $B$. It is clear that $B$ is a unital,
abelian C$^*$-algebra. It is known that there exists a period-2 conjugate-linear
$^*$-automorphism $\tau : B\to B$ such that $A = B^{\tau} := \{ x\in B : \tau (x) =x\}$
(cf. \cite[4.1.13]{Rick} and \cite[15.4]{Goo} or \cite[\S 5.2]{Li}).\smallskip

By the commutative Gelfand theory, there exists a compact Hausdorff space $K$ such that $B$ is
C$^*$-isomorphic to the C$^*$-algebra $C(K)$ of all complex valued continuous functions on $K$.
The Banach-Stone Theorem implies the existence of a homeomorphism $\sigma: K \to K$
such that $\sigma^2 (t) = t$, and $$ \tau (a) (t) = \overline{a (\sigma (t))},$$ for all $t\in
K$, $a\in C (K)$. Real function algebras of the form $C(K)^{\tau}$ have been studied by its own right
and are interesting in some other settings (cf. \cite{KulLim}).\smallskip

Henceforth, the symbol $\mathfrak{B}$ will stand for the
$\sigma$-algebra of all Borel subsets of $K$, $S(K)$ will denote
the space of $\mathfrak{B}$-simple scalar functions defined on $K$,
while the \emph{Borel algebra over $K$}, $B(K)$, is defined as the
completion of $S(K)$ under the supremum norm. It is known that
$B=C(K) \subset B(K) \subset C(K)^{**}.$ The mapping $\tau^{**}:
C(K)^{**} \to C(K)^{**}$ is a period-2 conjugate-linear
$^*$-automorphism on $B^{**}= C(K)^{**}.$ It is easy to see that
$\tau^{**} (B(K)) = B(K),$ and hence $\tau^{**}|_{B(K)} : B(K) \to
B(K)$ defines a period-2 conjugate-linear $^*$-automorphism on
$B(K)$.
By an abuse of notation, the symbol $\tau$ will denote
$\tau$, $\tau^{**}$ and $\tau^{**}|_{B(K)}$ indistinctly.
It is clear that, for each Borel set $B\in \mathfrak{B},$
$\tau (\chi_{_{B}}) = \chi_{_{\sigma(B)}}$.

Let $a$ be an element in $B(K)$. For each $\varepsilon>0,$
there exist complex numbers $\lambda_1,\ldots,\lambda_r$ and disjoint
Borel sets $B_1,\ldots,B_r$ such that $\displaystyle{\left\|a-\sum_{k=1}^r \lambda_k
 \chi_{_{B_k}} \right\|<{\varepsilon}.}$ When $a\in A$ is $\tau$-symmetric (i.e. $\tau(a)=a$)
 then, since $a=\frac{1}{2}(a+\tau(a)),$ we have $$\left\|a-\frac12 \sum_{k=1}^r \lambda_k
 \chi_{_{B_k}} + \overline{\lambda_k}
 \chi_{_{\sigma(B_k)}}\right\|\leq \frac12 \left\|a-\sum_{k=1}^r \lambda_k
 \chi_{_{B_k}} \right\| + \frac12 \left\|a-\sum_{k=1}^r \overline{\lambda_k}
 \chi_{_{\sigma(B_k)}} \right\|$$ $$\leq \frac12 \left\|a-\sum_{k=1}^r \lambda_k
 \chi_{_{B_k}} \right\| + \frac12 \left\|\tau\left(a-\sum_{k=1}^r \lambda_k
 \chi_{_{B_k}}\right) \right\| <\varepsilon.$$
Consequently, every element in $B(K)^\tau$ can be approximated in norm by finite linear combinations
of the form $\sum_k \alpha_k \chi_{_{B_k}}+\overline{\alpha_k} \chi_{_{\sigma(B_k)}},$ where $\alpha_1,\ldots,\alpha_n$ are complex numbers and $B_1,\ldots,B_n$ are mutually disjoint Borel sets.
Having in mind that, for each Borel set $B\in \mathfrak{B}$ and each $\alpha\in \mathbb{C},$ $\left(\alpha \chi_{_{B}}+\overline{\alpha} \chi_{_{\sigma(B)}} \right)^*= \overline{\alpha} \chi_{_{B}}+ {\alpha} \chi_{_{\sigma(B)}},$ we have $$\left(\alpha \chi_{_{B}}+\overline{\alpha} \chi_{_{\sigma(B)}} \right) + \left(\alpha \chi_{_{B}}+\overline{\alpha} \chi_{_{\sigma(B)}} \right)^* = 2 \Re\hbox{e} (\alpha) \left(2 \chi_{_{\sigma(B)\cap B}} + \chi_{_{\sigma(B)\backslash B}} +\chi_{_{B\backslash \sigma(B)}}\right)$$
$$= 2 \Re\hbox{e} (\alpha) \left(2 \chi_{_{\sigma(B)\cap B}} + \chi_{_{(\sigma(B)\backslash B)\cup \sigma(\sigma(B)\backslash B)}} \right),$$ and $$\left(\alpha \chi_{_{B}}+\overline{\alpha} \chi_{_{\sigma(B)}} \right) - \left(\alpha \chi_{_{B}}+\overline{\alpha} \chi_{_{\sigma(B)}} \right)^* = 2  i  \Im\hbox{m} (\alpha) \left( \chi_{_{B\backslash \sigma(B)}} - \chi_{_{\sigma(B)\backslash B}}  \right).$$

Suppose now that $a\in B(K)^\tau$ is *-symmetric (i.e. $a^*=a$). It follows from the above that $a$ can be approximated in norm by linear combinations of the form $\displaystyle{\sum_{k=1}^{r} \alpha_k \chi_{_{E_k}},}$ where $\alpha_k \in \mathbb{R}$ and $E_1,\ldots, E_r$ are mutually disjoint Borel subsets of $K$ with $\sigma(E_i) = E_i$. Let $b$ be an element in $B(K)^\tau$ satisfying $b^*=-b$. Similar arguments to those given for *-symmetric elements, allow us to show that $b$ can be approximated in norm by finite linear combinations of the form $\displaystyle{ \sum_{k=1}^{r}  i  \ \alpha_k (\chi_{_{E_k}}-\chi_{_{\sigma(E_k)}}),}$ where $\alpha_k \in \mathbb{R}$ and $E_1,\ldots, E_r$ are mutually disjoint Borel subsets of $K$ with $\sigma(E_i) \cap E_i = \emptyset$.

\begin{lemma}\label{l spectral resolution for skew symmetric 1} Let $A$ be a unital, abelian, real C$^*$-algebra whose complexification is denoted by $B= C(K)$, for a suitable compact Hausdorff space $K$. Let  $\tau : B\to B$ be a period-2 conjugate-linear $^*$-automorphism satisfying $A = B^{\tau}$ and $\tau (a) (t) =  \overline{a (\sigma (t))},$ for all $t\in K$, $a\in C (K)$, where  $\sigma: K \to K$ is a period-2 homeomorphism. Then the set $N=\{t\in K : \sigma (t) \neq t\}$ is an open subset of $K$, $F=\{t\in K : \sigma (t) = t\}$ is a closed subset of $K$ and there exists an open subset $\mathcal{O}\subset K$ maximal with respect to the property $\mathcal{O}\cap \sigma (\mathcal{O}) = \emptyset.$

\end{lemma}

\begin{proof}
That $F$ is closed follows easily from the continuity of $\sigma,$ and consequently, $N=K/F$ is open.

Let $\mathcal{F}$ be the family of all open subsets $O\subseteq K$ such that $O\cap \sigma( O)=\emptyset$ ordered by inclusion. Let $S=\{O_{\lambda}\}_{\lambda}$ be a totally ordered subset of $\mathcal{F}.$ We shall see that $O=\bigcup_{\lambda} O_{\lambda}$ is an open set which also lies in $\mathcal{F},$ that is, $O\cap \sigma (O)=\emptyset.$

Let us suppose, on the contrary, that there exists $t \in O \cap \sigma (O)\neq \emptyset.$ Then there exist $\lambda,\beta$ such that $t \in O_{\lambda}$ and $t \in \sigma (O_{\beta}).$  Since $S$ is totally ordered, $O_{\lambda} \subseteq O_{\beta}$ or $O_{\beta} \subseteq O_{\lambda}.$ We shall assume that $O_{\lambda} \subseteq O_{\beta}.$ Then $\sigma(O_{\lambda}) \subseteq \sigma(O_{\beta})$ and $t$ lies in $O_{\beta}\cap \sigma (O_{\beta})=\emptyset,$ which is a contradiction. Finally, Zorn's Lemma gives the existence of a maximal element $\mathcal{O}$ in $\mathcal{F}.$
\end{proof}

It should be noticed here that, in Lemma \ref{l spectral resolution for skew symmetric 1},
 $\mathcal{O}\cup \sigma(\mathcal{O}) = N$, an equality which follows from the maximality of $\mathcal{O}$.

Our next lemma analyses the ``spectral resolution'' of a *-skew-symmetric element in $B(K)^\tau.$

\begin{lemma}\label{l spectral resolution for skew symmetric} In the notation of Lemma \ref{l spectral resolution for skew symmetric 1}, let $B(A)= B(K)^{\tau}$, let $a\in B(K)^{\tau}_{sa},$ and let $b$ be an element in $B(A)_{skew}$. Then the following statements hold:\begin{enumerate}[$a)$]
\item  $b|F = 0$;
\item For each $\varepsilon >0$, there exist mutually disjoint Borel
sets $B_{1},\ldots,B_m \subset \mathcal{O}$ and real numbers $\lambda_1,\ldots,\lambda_m$ satisfying
$\displaystyle{\left\| b-\sum_{j=1}^{m}  i  \ \lambda_j (\chi_{_{B_j}} -  \chi_{_{\sigma(B_j)}})\right\| <\varepsilon;}$
\item For each $\varepsilon >0$, there exist mutually disjoint Borel
sets $C_{1},\ldots,C_m \subset K$ and real numbers $\mu_1,\ldots,\mu_m$ satisfying $\sigma (C_j)= C_j,$ and
$\displaystyle{\left\| a-\sum_{j=1}^{m} \mu_j \chi_{_{C_j}}\right\| <\varepsilon.}$
\end{enumerate}
\end{lemma}

\begin{proof}
$a)$ Since $b^*=-b$, we have $Re(b(t))=0,\forall t\in K.$ Now, let $t\in F$, applying $\sigma(t)=t$ and $\tau(b)=b$ we get
$\overline{b(t)}=\overline{b(\sigma(t))}=b(t),$ and hence $\Im\hbox{m}(b(t))=0.$

Statements $b)$ and $c)$ follow from the comments prior to Lemma \ref{l spectral resolution for skew symmetric 1} and the maximality of $\mathcal{O}$ in that Lemma.
\end{proof}

It is clear that in a commutative real (or complex) C$^*$-algebra, $A,$ two elements $a,b$ are
orthogonal if and only if they have zero-product, that is, $ab=0.$
Therefore, $V(a,b^*)=0=V(a,b)$ whenever $V:A \times A\to\RR$ is an orthogonal bilinear
form on an abelian real C$^*$-algebra and $a,b$ are two orthogonal elements in $A$.
We shall make use of this property without an explicit mention.

We shall keep the notation of Lemma \ref{l spectral resolution for skew symmetric 1} throughout the section.
Henceforth, for each $C\subseteq \mathcal{O}$ we shall write $u_{_C}= i  \ (\chi_{_C}-\chi_{_{\sigma(C)}}).$ The symbol $u_{_0}$ will stand for the element $u_{_{\mathcal{O}}}.$ It is easy to check
$1=\chi_{_F}+u_{_0}u_{_0}^*,$ where $1$ is the unit element in $B(K)^{\tau}.$
By Lemma \ref{l spectral resolution for skew symmetric} $a)$, for each  $b\in B(K)^{\tau}_{skew}$ we have
 $b \perp \chi_{_F},$ and so $b = b u_{_0}u_{_0}^*.$

\begin{proposition}\label{p extn Boral algebra orthogonal}
Let $K$ be a compact Hausdorff space, $\tau$
a period-2 conjugate-linear isometric $^*$-homomorphism on $C(K)$,
$A= C(K)^{\tau}$, and $V: A \times A \to \mathbb{R}$ be an orthogonal bounded
bilinear form whose Arens extension is denoted by $V^{**} :
A^{**}\times A^{**} \to \mathbb{R}$. Let $\sigma: K \to K$ be a period-2 homeomorphism
satisfying $\tau (a) (t) =  \overline{a (\sigma (t))},$ for all $t\in K$, $a\in C (K)$.
Then the following assertions hold for all Borel subsets $D,B,C$ of $K$ with
$\sigma(B)\cap B=\sigma(C) \cap C = \emptyset$ and $\sigma(D) =D$:
\begin{enumerate}[$a)$]
\item $V(\chi_{_D},u_{_B})=V(u_{_B},\chi_D)=0,$ whenever $D\cap
B=\emptyset$;

\item $V(u_{_B},u_{_C})=0,$ whenever $B\cap C=\emptyset;$

\item
$V((u_{_0}u_{_0}^*-u_{_C}u^*_{_C})u_{_B},u_{_C})=V(u_{_C},(u_{_0}u_{_0}^*-u_{_C}u^*_{_C})u_{_B})=0.$
\end{enumerate}
\end{proposition}

\begin{proof} By an abuse of notation, we write $V$ for $V$ and $V^{**}$.

Let $K_1,K_2$ be compact subsets of $K$ such that $K_1, K_2$ and $\sigma(K_2)$ are mutually disjoint.
By regularity and Urysohn's Lemma there exist nets $(f_{\lambda})_{_{\lambda}}$, $(g_{\gamma})_{_{\gamma}}$ in $C(K)^{+}$ such that $\chi_{_{K_1}}\leq f_{\lambda} \leq \chi_{_{K\backslash(K_2\cup\sigma(K_2))}},$ $\chi_{_{K_2}}\leq g_{\gamma} \leq \chi_{_{K\backslash(K_1\cup\sigma(K_1)\cup \sigma(K_2))}},$ $(f_{\lambda})_{_{\lambda}}$ (respectively, $(g_{\gamma})_{_{\gamma}}$) converges to $\chi_{_{K_1}}$ (resp., to $\chi_{_{K_2}}$) in the weak$^*$ topology of $C(K)^{**}$.

The nets $\widetilde{f}_{\lambda}=\frac{1}{2}(f_{\lambda}+\tau(f_{\lambda}))$ and $\widetilde{g}_{\gamma}=i(g_{\gamma}-\tau(g_\gamma))$ lie in $C(K)^{\tau}$ and converge in the weak$^*$ topology of $C(K)^{**}$ to $\frac{1}{2}(\chi_{_{K_1}}+\chi_{_{\sigma(K_1)}})$ and $u_{_{K_2}},$ respectively. It is also clear that $f_{\lambda} \perp g_{\gamma},$  $\tau(f_{\lambda}) \perp g_{\gamma},$ and hence $\widetilde{f}_{\lambda}\perp \widetilde{g}_{\gamma},$ for every $\lambda,\gamma.$

By the separate weak$^*$ continuity of $V^{**}\equiv V$ we have
 \begin{equation}\label{eq orth sym vs anti compact} V\left(\frac{1}{2}(\chi_{_{K_1}}+\chi_{_{\sigma(K_1)}}),u_{_{K_2}}\right)=w^*-\lim_{\lambda}\left( w^*-\lim_{\gamma}V\left(
\widetilde{f_{\lambda}},\widetilde{g}_{\gamma}\right)\right)=0, \end{equation}
and $$V\left(u_{_{K_2}},\frac{1}{2}(\chi_{_{K_1}}+\chi_{_{\sigma(K_1)}})\right)=0. $$

We can similarly prove that \begin{equation}\label{eq orth anti vs anti compact} V\left(u_{_{K_1}},u_{_{K_2}}\right)=0,
 \end{equation} whenever $K_1$ and $K_2$ are two compact subsets of $K$ such that $K_1, K_2, \sigma(K_1)$ and $\sigma(K_2)$ are pairwise disjoint.

$a)$ Let now $D,B$ be two disjoint Borel subsets of $K$ such that $\sigma(D)=D$ and $B\subseteq \mathcal{O}.$
By inner regularity there exist nets of the form $(\chi_{_{K_{\lambda}^{^D}}})_{\lambda}$ and $(\chi_{_{K_{\gamma}^{^B}}})_{\gamma}$ such that $(\chi_{_{K_{\lambda}^{^D}}})_{\lambda}$ and  $(\chi_{_{K_{\gamma}^{^B}}})_{\gamma}$ converge in the weak$^*$ topology of $C(K)^{**}$ to $\chi_{_D}$ and $\chi_{_B}$, respectively, where each $K_{\lambda}^{^D} \subseteq D$ and each $K_{\gamma}^{^B}\subseteq B$ is
compact subset of $K$. By the assumptions made on $D$ and $B$ we have that $K_{\lambda}^{^D}\cap K_{\gamma}^{^B}=K_{\lambda}^{^D}\cap \sigma(K_{\gamma}^{^B})=\emptyset$ and $K_{\gamma}^{^B}\subseteq \mathcal{O}$ for all $\lambda$ and $\gamma.$ By (\ref{eq orth sym vs anti compact}) and the separate weak$^*$ continuity of $V$ we have \begin{equation}
 \label{eq orth symm vs anti} V(\chi_{_D},u_{_B})=w^*-\lim_{\lambda} \left(w^*-\lim_{\gamma} V\left(\frac{\chi_{_{K_{\lambda}^{^D}}}+\chi_{_{\sigma(K_{\lambda}^{^D})}}}{2},
  u_{_{K_{\gamma}^{^B}}}\right)\right)=0,
 \end{equation}
 and
\begin{equation}
 \label{eq orth antivs symm} V(u_{_B},\chi_{_{_D}})=0.
 \end{equation}\medskip

A similar argument, but replacing (\ref{eq orth sym vs anti compact}) with (\ref{eq orth anti vs anti compact}), applies to prove $b)$.

To prove the last statement, we observe that
$$(u_{_0}u_{_0}^*-u_cu_c^*)u_{_B}=(\chi_{_\mathcal{O}}+\chi_{_{\sigma(\mathcal{O}})}-\chi_{_C}-\chi_{_{\sigma(C)}})u_{_B}=
  (\chi_{_{\mathcal{O}\setminus C}}+ \chi_{_{\sigma(\mathcal{O}\setminus C)}})u_{_B}=u_{_{(\mathcal{O}\setminus C)\cap
  B}},$$ and hence the statement $c)$ follows from $b)$.
\end{proof}

We can now establish the description of all orthogonal forms on a commutative real C$^*$-algebra.

\begin{theorem}\label{t orhtogonal forms real abelian}
Let $V: A\times A \to \mathbb{R}$ be a continuous orthogonal form on a commutative real C$^*$-algebra,
then there exist $\varphi_1,$ and $\varphi_2$ in $A^{*}$ satisfying $$V(x,y) = \varphi_1 (x y) + \varphi_2 ( x y^*),$$ for every $x,y\in A.$
\end{theorem}

\begin{proof} We may assume, without loss of generality, that $A$ is unital (compare Proposition \ref{prop extns to the multiplier}). Let $B$ denote the complexification of $A$. In this case $B$ identifies with $C(K)$ for a suitable compact Hausdorff space $K$ and $A=C(K)^{\tau}$, where $\tau$ is a conjugate-linear period-2 *-homomorphism on $C(K)$. We shall follow the notation employed in the rest of this section.

The form $V: A\times A \to \mathbb{R}$ extends to a continuous form $V^{**}: A^{**}\times A^{**} \to \mathbb{R}$ which is separately weak$^*$ continuous (cf. Lemma \ref{l extensions real forms}). The restriction  $V^{**}|_{B(K)^{\tau}\times B(K)^{\tau}}: B(K)^{\tau}\times B(K)^{\tau} \to \mathbb{R}$ also is a continuous extension of $V$. We shall prove the statement for $V^{**}|_{B(K)^{\tau}\times B(K)^{\tau}}.$ Henceforth, the symbol $V$ will stand for $V$, $V^{**}$ and $V^{**}|_{B(K)^{\tau}\times B(K)^{\tau}}$ indistinctly.

Let us first take two self-adjoint elements $a_1,a_2$ in $B(K)^{\tau}$.
By Proposition \ref{p orthog forms on Asa}, \begin{equation}
\label{eq theorem -1} V(a_1,a_2)=V(a_1 a_2,1).
\end{equation}

To deal with the skew-symmetric part, let $D,B,C$ be Borel subsets of $K$ with, $D=\sigma(D)$ and $B,C \subseteq
\mathcal{O}.$ From Proposition \ref{p extn Boral algebra orthogonal} $a)$, we have \begin{equation}\label{eq sym against anti}V(\chi_{_D},u_{_B})=V(\chi_{_D},u_{_B}(1-\chi_{_D}+\chi_{_D}))= V(\chi_{_D},u_{_{B\cap (K\backslash D)}})+ V(\chi_{_D},u_{_B}\chi_{_D})\end{equation} $$=V(\chi_{_D}-1+1,u_{_B}\chi_{_D})= V(-\chi_{_{(K\backslash D)}}+1,u_{_{(B\cap D)}}) =V(1,u_{_B}\chi_{_D}).$$

Similarly,
\begin{equation}\label{eq anti against sym}V(u_{_B},\chi_{_D})=V(u_{_B}\chi_{_D},1).
\end{equation}

Now, Proposition \ref{p extn Boral algebra orthogonal} $b)$ and $c)$, repeatedly applied give:  $$V(u_{_B},u_{_C})=V(u_{_B}(\chi_{_F}+u_{_0}u_{_0}^*),u_{_C})=V(u_{_B}u_{_0}u_{_0}^*,u_{_C})$$ $$=V(u_{_B}(u_{_0}u_{_0}^*+u_{_C}u^*_{_C} -u_{_C}u^*_{_C}),u_{_C})=   V(u_{_B}u_{_C}u^*_{_C},u_{_C})$$ $$=V(u_{_B}u_{_C}u^*_{_C},u_{_C}-u_{_0}+u_{_0}) = V(u_{_{(B\cap C)}},-u_{_{(\mathcal{O}\backslash C)}}+u_{_0}) =V(u_{_{(B\cap C)}},u_{_0})$$ $$= V(u_{_B}u_{_C}(u^*_{_C}-u_{_0}^*+u_{_0}^*),u_{_0})=V(u_{_B}u_{_C}u_{_0}^*,u_{_0}).$$
Thus, we have \begin{equation} \label{eq anti anti 1}V(u_{_B},u_{_C})=V(u_{_B}u_{_C}u_{_0}^*,u_{_0}),
\end{equation} and similarly \begin{equation} \label{eq anti anti 2}V(u_{_B},u_{_C})= V(u_{_0},u_{_B}u_{_C}u_{_0}^*).\end{equation}

Let $\displaystyle{a_l=\sum_{j=1}^{m_l} \mu_{l,j} \chi_{_{D^l_j}},}$ $\displaystyle{b_l=\sum_{k=1}^{p_l} \lambda_{l,k} u_{_{B^l_k}}}$ ($l\in\{1,2\}$) be two simple elements in $B(K)^{\tau}_{sa}$ and $B(K)^{\tau}_{skew}$, respectively, where $\lambda_{l,k}, \mu_{l,j}\in \mathbb{R},$ for each $l\in \{1,2\},$ $\{D^l_1,\ldots, D^l_{m_l}\}$ and $\{B^l_1,\ldots, B^l_{p_l}\}$ are families of mutually disjoint Borel subsets of $K$ with $\sigma(D^l_j) = D^l_j$ and $B^l_i \subseteq \mathcal{O}.$ By $(\ref{eq theorem -1})$, $(\ref{eq sym against anti})$, $(\ref{eq anti against sym})$, and $(\ref{eq anti anti 1}),$ we have
$$V(a_1+b_1, a_2+b_2)=V(a_1 a_2,1)+ \sum_{j=1}^{m_1} \sum_{k=1}^{p_2} \mu_{1,j}  \lambda_{2,k} V\left( \chi_{_{D^1_j}},  u_{_{B^2_k}}\right)$$ $$+\sum_{k=1}^{p_1}\sum_{j=1}^{m_2} \mu_{2,j}  \lambda_{1,k} V\left( u_{_{B^1_k}}, \chi_{_{D^2_j}}\right)+\sum_{k=1}^{p_1}\sum_{k=1}^{p_2} \lambda_{2,k} \lambda_{1,k} V\left( u_{_{B^1_k}},  u_{_{B^2_k}}\right)$$ $$= V(a_1 a_2,1)+ \sum_{j=1}^{m_1} \sum_{k=1}^{p_2} \mu_{1,j}  \lambda_{2,k} V\left(1 ,  \chi_{_{D^1_j}} u_{_{B^2_k}}\right)$$ $$+\sum_{k=1}^{p_1}\sum_{j=1}^{m_2} \mu_{2,j}  \lambda_{1,k} V\left( u_{_{B^1_k}} \chi_{_{D^2_j}} , 1\right)+\sum_{k=1}^{p_1}\sum_{k=1}^{p_2} \lambda_{2,k} \lambda_{1,k} V\left( u_{_{B^1_k}}u_{_{B^2_k}} u_{_0}^*, u_{_0} \right) $$ $$= V(a_1 a_2,1)+  V\left(1 ,  a_1 b_2\right) +V\left( b_1 a_2,1\right) + V\left( b_1 b_2 u_{_0}^*, u_{_0}\right)$$
$$= \psi_1 (a_1 a_2)+  \psi_2 \left(a_1 b_2\right) +\psi_1 \left( b_1 a_2 \right) + \psi_4 \left( b_1 b_2\right),$$
where $\psi_1,\psi_2,$ and $\psi_4$ are the functionals in $A^*$ defined by $\psi_1(x)= V(x,1),$ $\psi_2(x)= V(1,x),$ and $\psi_4(x) = V(x u_{_0}^*,u_{_0}),$ respectively. Since, by Proposition \ref{l spectral resolution for skew symmetric}, simple elements of the above form are norm-dense in $B(K)^{\tau}_{sa}$ and $B(K)^{\tau}_{skew}$, respectively, and $V$ is continuous, we deduce that $$V(a_1+b_1, a_2+b_2)= \psi_1 (a_1 a_2)+  \psi_2 \left(a_1 b_2\right) +\psi_1 \left( b_1 a_2 \right) + \psi_4 \left( b_1 b_2\right),$$ for every $a_1,a_2\in B(K)^{\tau}_{sa}$, $b_1,b_2 \in B(K)^{\tau}_{skew}.$

Now, taking $\phi_1 = \frac14 (2 \psi_1 +\psi_2+\psi_4),$ $\phi_2 = \frac14 (2 \psi_1 -\psi_2-\psi_4),$
$\phi_3 = \frac14 ( \psi_2- \psi_4),$ and $\phi_4 = \frac14 (\psi_4-\psi_2),$ we get $$V(a_1+b_1, a_2+b_2)= \phi_1 ((a_1+b_1) (a_2+b_2))+  \phi_2 \left((a_1+b_1) (a_2+b_2)^*\right) $$ $$+\phi_3 \left( (a_1+b_1)^* (a_2+b_2) \right) + \phi_4 \left( (a_1+b_1)^* (a_2+b_2)^*\right),$$ for every $a_1,a_2\in B(K)^{\tau}_{sa}$, $b_1,b_2 \in B(K)^{\tau}_{skew}.$

Finally, defining $\varphi_1 (x) = \phi_1 (x)+ \phi_4(x^*)$ and  $\varphi_2(x)= \phi_2(x)+ \phi_3(x^*)$ $(x\in A)$, we get the desired statement.
\end{proof}

\begin{remark}\label{r uniquenes  functionals}
The functionals $\varphi_1$ and $\varphi_2$ appearing in Theorem \ref{t orhtogonal forms real abelian} need not be unique. For example, let $(\varphi_1,\varphi_2)$ and $(\phi_1,\phi_2)$ be two couples of elements in the dual of a commutative real C$^*$-algebra $A$. It is not hard to check that $$\varphi_1 (x y) + \varphi_2 ( x y^*) = \phi_1 (x y) + \phi_2 ( x y^*),$$ for every $x,y\in A$ if, and only if, $\varphi_1 +\varphi_2 = \phi_1+\phi_2$, $(\varphi_1 -\varphi_2) (z) = (\phi_1-\phi_2) (z)$ and $(\varphi_1 -\varphi_2) (z w) = (\phi_1-\phi_2) (z w),$ for every $z,w\in A_{skew}.$ These conditions are not enough to guarantee that $\phi_i=\varphi_i$. Take, for example, $A= \mathbb{R}\oplus^{\infty} \mathbb{C}_{\mathbb{R}},$ $\phi_1 (a,b) = a + \Re\hbox{e} (b) + \Im\hbox{m} (b),$  $\phi_2 (a,b) = 0,$ $\varphi_1 (a,b) = \frac{a}{2} + \Re\hbox{e} (b) + \Im\hbox{m} (b),$ and $\varphi_2 (a,b) = \frac{a}{2} .$
\end{remark}

\begin{corollary}\label{c orhtogonal forms bidual}
Let $V: A\times A \to \mathbb{R}$ be a continuous orthogonal form on a commutative real C$^*$-algebra, then
its (unique) Arens extension $V^{**}: A^{**}\times A^{**} \to \mathbb{R}$ is an orthogonal form.$\hfill\Box$
\end{corollary}

Clearly, the statement of the above Theorem \ref{t orhtogonal forms real abelian} doesn't hold for bilinear forms on a commutative (complex) C$^*$-algebra. The real version established in this paper is completely independent to the result proved by K. Ylinen for commutative complex C$^*$-algebras in \cite{Yl75} and \cite{Gold}. It seems natural to ask whether the real result follows from the complex one by a mere argument of complexification. Our next example shows that the (canonical) extension of an orthogonal form on a commutative real C$^*$-algebra need not be an orthogonal form on the complexification.

\begin{example}\label{exam complexification} Let
$K=\{t_1,t_2\}.$ We define $\sigma:K\rightarrow K$ by
$\sigma(t_1)=t_2.$ Let $A=C(K)^{\tau}$ be the real C$^*$-algebra whose complexification is $C(K)$ and let $V:A\times A\to \mathbb{R},$ be the orthogonal form defined by $V(x,y)=\phi_{_{t_1}}(xy^*) =\Re\hbox{e}(x(t_1)\overline{y(t_1)})=\Re\hbox{e}(x(t_1)y(t_2)), $ where $\phi_{t_1} =\Re\hbox{e}(\delta_{_{t_1}}).$ In this case, the canonical complex bilinear extension
$\widetilde{V} : C(K)\times C(K) \to \mathbb{C}$ is given by $\widetilde{V} (x,y)= \phi_{_{t_1}} (x\tau(y)^*)=x(t_1)y(t_2)$ $(x,y\in C(K)).$ It is clear that $\chi_{_{t_1}}\perp \chi_{_{t_2}}$ in $C(K),$ however $\widetilde{V} (\chi_{_{t_1}},\chi_{_{t_2}})=1\neq 0,$ which implies that $\widetilde{V}$ is not orthogonal.
\end{example}

The (complex) bilinear extension of an orthogonal form $V$ on a real C$^*$-algebra to its complexification is orthogonal
precisely when $V$ satisfies the generic form of an orthogonal form on a (complex) C$^*$-algebra given by the main result in \cite{Gold}.

\begin{corollary}\label{c orhtogonal forms complexification}
Let $V: A\times A \to \mathbb{R}$ be a continuous orthogonal form on a commutative real C$^*$-algebra, let $B$ denote the
complexification of $A$ and let $\widetilde{V} : B\times B \to \mathbb{R}$ be the (complex) bilinear extension of $V$.
Then the form $\widetilde{V}$ is orthogonal if, and only if, $V$ writes in the form $V(x,y) = \varphi_1 (x y)$ $(x,y\in A),$
where $\varphi_1$ is a functional in $A^{*}$.
\end{corollary}

\begin{proof}
Let $\tau$ be the period-2 $^*$-automorphism on $B$ satisfying that $B^{\tau}= B$ and let $\widetilde{\tau}: B^{*} \to B^{*}$
be the involution defined by $\widetilde{\tau} (\phi) (b) = \overline{\phi(\tau(b))}$.

Suppose $\widetilde{V}$ is orthogonal. By the main result in \cite{Gold} (see also \cite{Yl75}), there exists $\phi\in B^*$ satisfying
$\widetilde{V}(x,y) = \phi (x y),$ for every $x,y\in B$. Since $\widetilde{V}$ is an extension of $V$, we get
$V(a,b) = \Re\hbox{e} \phi (a b)= \phi (a b),$ for every $a,b\in A$. In particular, $\phi (a) \in \mathbb{R}$, for every $a\in A$ and hence
$\widetilde{\tau} (\phi)= \phi$ lies in $(B^*)^{\widetilde{\tau}} \equiv A^*$.

Let us assume that $V$ writes in the form $V(x,y) = \varphi_1 (x y)$ $(x,y\in A)$, where $\varphi_1$
is a functional in $A^{*}$. The functional $\varphi_1$ can be regarded as an element in $B^*$ satisfying
$\widetilde{\tau} (\varphi_1)= \varphi_1.$ It is easy to check that $\widetilde{V}(x,y) = \varphi_1 (x y),$
for every $x,y\in B$.
\end{proof}

\section{Orthogonality preservers between commutative real C$^*$-algebras}

Throughout this section, $A_1 = C(K_1)^{\tau_1}$ and $A_2 = C(K_2)^{\tau_2}$ will denote two unital commutative real C$^*$-algebras, $K_1$ and $K_2$ will be two compact Hausdorff spaces and $\tau_i$ will denote a conjugate-linear period-2 $^*$-automorphism on $C(K_i)$ given by $\tau_i (f) (t)= \overline{f(\sigma_i (t))}$ ($t\in K_i$, $f\in C(K_i)$),  where $\sigma_i : K_i \to K_i$ is a period-2 homeomorphism. We shall write $B_1 = C(K_1)$ and $B_2 = C(K_2)$ for the corresponding complexifications of $A_1 $ and $A_2$, respectively.

By the Banach-Stone theorem, every surjective isometry $T:C(K_1) \to C(K_2)$ is a composition operator, that is, there exist a unitary element $u$ in $C(K_2)$ and a homeomorphism $\sigma: K_2\to K_1$ such that $T(f)
(t) = (uC_{\sigma})(f)(t) := u(t) \ f (\sigma (t))$ ($t\in K_2$, $f\in C(K_1)$). This result led to the study of the so-called Banach-Stone theorems in different classes of Banach spaces containing $C(K)$-spaces, in which their algebraic and geometric properties are mutually determined. That is the case of general C$^*$-algebras (R. Kadison \cite{Kad} and Paterson and Sinclair \cite{PaSi}), JB- and JB$^*$-algebras (Wright and M. Youngson \cite{WriYo} and Isidro and A. Rodr\'{\i}guez \cite{IsRo95}), JB$^*$-triples (Kaup \cite{Ka} and Dang, Friedman and Russo \cite{DaFriRu}), real C$^*$-algebras (Grzesiak \cite{Grez}, Kulkarni and Arundhathi \cite{KulAn}, Kulkarni and Limaye \cite{KulLim} and  Chu, Dang, Russo and Ventura \cite{ChuDaRuVen}) and real JB$^*$-triples (Isidro, Kaup and Rodríguez \cite{IsKaRo95}, Kaup \cite{Ka97} and Fernández-Polo, Martínez and Peralta \cite{FerMarPe04}). In what concerns us, we highlight that any surjective linear isometry $T: C(K_1)^{\tau_1}\to C(K_2)^{\tau_2}$ is a composition operator given by a homeomorphism $\phi: K_2 \to K_1$ which satisfies $\sigma_1 \circ \phi=\phi\circ \sigma_2$ (cf. \cite{Grez} or \cite{KulAn} or  \cite[Corollary 5.2.4]{KulLim}).

The class of orthogonality preserving (continuous) operators between $C(K)$-spaces is strictly bigger than the class of surjective isometries. Actually, a bounded linear operator $T: C(K_1) \to C(K_2)$ is orthogonality preserving (equivalently, disjointness preserving) if, and only if, there exist $u$ in $C(K_2)$ and a mapping $\varphi: K_2\to K_1$ which is continuous on $\{t\in K_2 : u(t)\neq 0\}$ such that $T(f) (t)=(uC_{\varphi})(f)(t) = u(t) \ f (\varphi (t))$ (compare \cite[Example 2.2.1]{Arendt}).

Developing ideas given by E. Beckenstein, L. Narici, and A.R. Todd in \cite{BeckNarTodd} and \cite{BeckNarTodd88} (see also \cite{BeckNar}), K. Jarosz showed, in \cite{Jar}, that the above hypothesis of $T$ being continuous can be, in some sense, relaxed. More concretely, for every orthogonality preserving linear mapping $T: C(K_1) \to C(K_2)$, there exists a disjoint decomposition $K_2=S_1\cup S_2\cup S_3$ (with $S_2$ open, $S_3$ closed), and a continuous mapping $\varphi$ from $S_1\cup S_2$ into $K_1$ such that $T(f)(s)=\chi (s) f(\varphi(s))$ for all $s\in S_1$ (where $\chi$ is a continuous, bounded, non-vanishing, scalar-valued function on $S_1$), $T(f)(s)=0$  for all $s\in S_3$, $\varphi(S_2)$ is finite and, for each $s\in S_2$, the mapping $f\mapsto T(f)(s)$ is not continuous. As a consequence, every orthogonality preserving linear bijection between $C(K)$-spaces is (automatically) continuous. More recently, M. Burgos and the authors of this note prove, in \cite{BurGarPe}, that every bi-orthogonality preserving linear surjection between two von Neumann algebras (or between two compact C$^*$-algebras) is automatically continuous (compare \cite{OikPeRa}, \cite{OikPePu} for recent additional generalisations).

The main goal of this section is to describe the orthogonality
preserving linear mappings between $C(K)^{\tau}$-spaces. Among the consequences, we
establish that every orthogonality preserving linear bijection
between unital commutative real C$^*$-algebras is automatically
continuous. We shall provide an example of an orthogonality
preserving linear bijection between $C(K)^{\tau}$-spaces which is
not bi-orthogonality preserving
 and give a
characterisation of bi-orthogonality preserving linear maps.

We shall borrow and adapt some of the ideas developed in those previously mentioned papers (cf. \cite{BeckNarTodd,BeckNarTodd88} and \cite{Jar}). In order to have a good balance between completeness and conciseness, we just give some sketch of the refinements needed in our setting. In any case, the results presented here are independent innovations and extensions of those proved by Beckenstein, Narici, and Todd and Jarosz for $C(K)$-spaces.

Let $T:C(K_1)^{\tau_1} \to C(K_2)^{\tau_2}$ be an orthogonality preserving linear mapping. Keeping in mind the notation in the previous section, we write $L_i := \mathcal{O}_i\cup F_i,$ where $\mathcal{O}_i$ and $F_i$ are the subsets of $K_i$ given by Lemma \ref{l spectral resolution for skew symmetric 1}. The map sending each $f$ in  $C(Ki)^{\tau_i}$ to its restriction to $L_i$ is a C$^*$-isomorphism (and hence a surjective linear isometry) from
$C(Ki)^{\tau_i}$ onto the real C$^*$-algebra $C_r(L_i)$ of all continuous functions $f : L_i \to \mathbb{C}$ taking real values on $F_i$. Thus, studying orthogonality preserving linear maps between $C(K)^{\tau}$ spaces is equivalent to study orthogonality preserving linear mappings between the corresponding $C_r(L)$-spaces.

Henceforth, we consider an orthogonality preserving (not necessarily continuous) linear map $T:C_r(L_1) \to C_r(L_2)$, where $L_1$ and $L_2$ are two compact Hausdorff spaces and each $F_i$ is a closed subset of $L_i$. Let us consider the sets  $$Z_1= \{ s\in L_2 : {\delta}_{s} T \hbox{ is a non-zero bounded real-linear mapping} \},$$ $$Z_3 = \{s\in L_2 : {\delta}_{s} T =0 \}, \hbox{ and } Z_2 = L_2\backslash (Z_1\cup Z_3).$$ It is easy to see that $Z_3$ is closed. Following a very usual technique (see, for example, \cite{BeckNarTodd,BeckNarTodd88,Jar,Dubarbie10} and \cite{Dubarbie10b}), we can define a continuous support map $\varphi : Z_1\cup  Z_2 \to L_1$. More concretely, for each $s\in Z_1\cup  Z_2$, we write $\hbox{supp} (\delta_s T)$ for the set of all $t\in L_1$ such that for each open set $U\subseteq L_1$ with  $t\in U $ there exists $f\in C_r(L_1)$ with $\hbox{coz}(f)\subseteq U$ and $\delta_{s} (T(f)) \neq 0$. Actually, following a standard argument, it can be shown that, for each $s\in Z_1\cup  Z_2,$ $\hbox{supp} (\delta_s T)$ is non-empty and reduces exactly to one point $\varphi(s)\in L_1$, and the assignment $s\mapsto \varphi(s)$ defines a continuous map from $Z_1\cup  Z_2$ to $L_1$. Furthermore, the value of $T(f)$ at every $s\in  Z_1$ depends strictly on the value $f(\varphi(s))$. More precisely, for each $s\in Z_1$ with $\varphi(s)\notin F_1$, the value $T(g) (s)$ is the same for every function $g\in C_r(L_1)$ with $g\equiv i$ on a neighborhood of $\varphi (s).$ Thus, defining $T(i) (s) :=0$ for every $s\in Z_3\cup  Z_2$ and for every $s\in Z_1$ with $\varphi(s)\in F_1$, and $T(i) (s):=T(g)(s)$ for every $s\in Z_1\cup  Z_2$ with $\varphi(s)\notin F_1$, where $g$ is any element in $C_r(L_1)$ with $g\equiv i$ on a neighborhood of $\varphi (s),$ we get a (well-defined) mapping $T(i): L_2 \to \mathbb{C}.$ It should be noticed that ``$T(i)$'' is just a symbol to denoted the above mapping and not an element in the image of $T$. In this setting, the identity $$T(f) (s) = T(1) (s) \ \Re\hbox{e} f(\varphi(s)) + T(i) (s) \ \Im\hbox{m} f(\varphi(s)),$$ holds for every $s\in  Z_1.$ Clearly, $T(1) (s), T(i) (s)\in \mathbb{R}$, for every $s\in F_2$ and $\left|T(1) (s)\right|+ \left|T(i) (s)\right| \neq 0$, for every $s\in Z_1$.

The following property also follows from the definition of $\varphi$ by standard arguments: Under the above conditions,
let $s$ be an element in $Z_1\cup Z_2,$ then \begin{equation}\label{eq varphi notin cozero} \delta_s T (f) =0 \hbox{ for every } f\in C_r(L_1) \hbox{ with } \varphi(s)\notin \overline{\hbox{coz}(f)}.
\end{equation}

\begin{lemma}\label{l 3.1} The mapping $T(i)$ is bounded on the set $\varphi^{-1} (\mathcal{O}_1)$. Furthermore, the inequality $$\left|T(f)(s)\right| \leq \|T(1)\| + \sup_{\widetilde{s}\in \varphi^{-1} (\mathcal{O}_1)} |T(i)(\widetilde{s})|$$ holds for all $s\in Z_1$ and all $f\in C_r(L_1)$ with $|\Re\hbox{e} (f)|,|\Im\hbox{m} (f)| \leq 1.$
\end{lemma}

\begin{proof} Arguing by contradiction, we suppose that, for each natural $n$, there
exists $s_n \in \varphi^{-1} (\mathcal{O}_1)$ such that $\left|T(i) (s_n)\right| > n^3$.
The elements $s_n's$ can be chosen so that $\varphi(s_n) \neq \varphi(s_m)$  for $n \neq m$, and consequently
we can find a sequence of pairwise disjoint open subsets $(U_n)$ of $\mathcal{O}_1$ with $\varphi(s_n) \in U_n$.
It is easily seen that we can define a function $\displaystyle g = \sum_{n=1}^{\infty} i \ g_n \in C_r(L_1)$ with coz$(g_n) \subset U_n$, $0\leq g_ n \leq \frac{1}{n^2}$, and $g_n \equiv \frac{1}{n^2}$ on
a neighborhood of $s_n$, for all $n$. By the form of $g$, and since $T$ is orthogonality preserving, we have $|T(g)(s_n)| = n^2 |T(i) (s_n)| > n$ for all $n$, which is absurd.
\end{proof}

We can easily show now that $Z_2$ is an open subset of $L_2$. With this aim, we consider an element $s_0$ in $Z_2$. We can pick a function $f\in C_r(L_1)$ such that $\|f\|\leq 1$ and $$|T(f) (s_0)| > 1 + \|T(1)\| + \sup_{\widetilde{s}\in \varphi^{-1} (\mathcal{O}_1)} |T(i)(\widetilde{s})|.$$ Since $\displaystyle\left|T(f)(s)\right| \leq \|T(1)\| + \sup_{\widetilde{s}\in \varphi^{-1} (\mathcal{O}_1)} |T(i)(\widetilde{s})|< |T(f) (s_0)| - 1$, for every $s\in Z_1\cup Z_3,$ we conclude that there exists an open neighborhood of $s_0$ contained in $Z_2$.

The next theorem resumes the above discussion.

\begin{theorem}\label{t op between Cr(K)} In the notation above, let $T: C_r(L_1)\to C_r(L_2)$ be an orthogonality preserving linear mapping. Then $L_2$ decomposes as the union of three mutually disjoint subsets $Z_1, Z_2,$ and $Z_3$, where $Z_2$ is open and $Z_3$ is closed, there exist a continuous support map $\varphi : Z_1\cup  Z_2 \to L_1$, and a bounded mapping $T(i) : L_2 \to \mathbb{C}$ which is continuous on $\varphi^{-1} (\mathcal{O}_1)$ satisfying: $$T(i) (s)\in \mathbb{R} \ (\forall s\in F_2), \ T(i) (s)=0, \ (\forall s\in Z_3\cup Z_2 \hbox{ and } \forall s\in Z_1 \hbox{ with } \varphi(s) \in F_1),$$
\begin{equation}
\label{eq 1 thm 32}\left|T(1) (s)\right|+ \left|T(i) (s)\right| \neq 0,\ (\forall s\in Z_1),
\end{equation}
\begin{equation}
\label{eq 2 thm 32} T(f) (s) = T(1) (s) \ \Re\hbox{e} f(\varphi(s)) + T(i) (s) \ \Im\hbox{m} f(\varphi(s)), \hbox{ {\rm(}$\forall s\in Z_1, f\in C_r(L_1){\rm)},$}
\end{equation}
$$T(f) (s)=0, \hbox{ {\rm(}$\forall s\in Z_3, f\in C_r(L_1)${\rm),}}$$ and for each $s\in L_2$, the mapping $C_r(L_1) \to \mathbb{C}$, $f\mapsto T(f(s)),$ is unbounded if, and only if, $s\in Z_2$. Furthermore, the set $\varphi (Z_2)$ is finite.
\end{theorem}

\begin{proof}
Everything has been substantiated except perhaps the statement concerning the set $\varphi (Z_2)$. Arguing by contradiction, we assume the existence of a sequence $(s_n)$ in $Z_2$ such that $\varphi (s_n)\neq \varphi (s_m)$ for every $n\neq m$. Find a sequence $(U_n)$ of mutually disjoint open subsets of $L_1$ satisfying $\varphi(s_n) \in U_n$ and a sequence $(f_n)\subseteq C_r(L_1)$ such that $\|f_n\|\leq \frac1n$, $\hbox{coz} (f_n)\subseteq U_n$ and $|\delta_{s_n} T(f_n)| >n$, for every $n\in \mathbb{N}$. The element $\displaystyle f= \sum_{n=1}^{\infty} f_n$ lies in $C_r(L_1),$ and for each natural $n_0$, $\displaystyle{f_{n_0} \perp \sum_{n=1,n\neq n_{0}}^{\infty} f_n}.$ Thus, $|\delta_{s_{n_0}} T(f)| \geq |\delta_{s_{n_0}} T(f_{n_0})|>n_0,$ which is impossible.
\end{proof}

\begin{remark}{\rm
The mapping $T(i) : L_2 \to \mathbb{C}$ has been defined to satisfy $T(i) (s)=0,$  for all $s\in Z_3\cup Z_2$ and for all $s\in Z_1$ with $\varphi(s) \in F_1.$
It should be noticed here that the value $T(i)(s)$ is uniquely determined only when $s\in Z_1$ and $\varphi(s) \notin F_1$. There are some other choices for the values of $T(i)(s)$ at $s\in Z_3\cup Z_2$ and at $s\in Z_1$ with $\varphi(s) \in F_1$ under which conditions $(\ref{eq 1 thm 32})$ and $(\ref{eq 2 thm 32})$ are satisfied.}
\end{remark}

\begin{remark}\label{r consequences}{\rm
We shall now explore some of the consequences derived from Theorem \ref{t op between Cr(K)}.
Let $T: C_r(L_1)\to C_r(L_2)$ be an orthogonality preserving linear mapping.
\begin{enumerate}[$(a)$] \item The set $Z_3$ is empty whenever $T$ is surjective.
\item $Z_3=\emptyset$ implies that $Z_1= L_2\backslash Z_2$ is a compact subset of $L_2.$
\item $\varphi(Z_2)$ is a finite set of non-isolated points in $L_1$. Indeed, if $\varphi(s_0)=t_0$ is isolated for some $s_0\in Z_2$, then we can find an open set $U\subseteq L_1$ such that $U\cap K_1= \{t_0\}$. Therefore, for each $f\in C_r(L_1)$ with $f(t_0)=0$ we have $\delta_{s_0} T(f) = 0$. Pick an arbitrary $h\in C_r(L_1)$. Clearly, $\chi_{_{t_0}}\in C_r (L_1)$, while $i  \chi_{_{t_0}} $ lies in $C_r(L_1)$ if, and only if, $t_0\notin F_1$. Therefore, $$h_0 =\Re\hbox{e} (h(t_0)) \chi_{_{t_0}} + \Im\hbox{m} (h(t_0))  \  i \chi_{_{t_0}} $$ lies in $C_r(L_1)$ and $(h-h_0)(t_0)=0$.

    Assume first that $t_0\notin F_1$. Denoting $\lambda_0 = \delta_{s_0} T(\chi_{_{t_0}})$ and $\mu_{_0} = \delta_{s_0} T( i \chi_{_{t_0}}),$ we have $$\delta_{s_0} T(h) = \delta_{s_0} T(h_0)= \lambda_{0} \Re\hbox{e} (h(t_0)) +\mu_{_0} \Im\hbox{m} (h(t_0))$$ $$= \frac{\lambda_0- i \mu_{_0}}{2} \delta_{t_0} (h)+ \frac{\lambda_0+ i \mu_{_0}}{2} \overline{\delta_{t_0}} (h).$$ This shows that $\delta_{s_0} T = \frac{\lambda_0- i \mu_{_0}}{2} \delta_{t_0} + \frac{\lambda_0+ i \mu_{_0}}{2} \overline{\delta_{t_0}}$ is a continuous mapping from $C_r(L_1)$ to $\mathbb{C}$, which is impossible.

    When $t_0\in F_1$ we have $\delta_{s_0} T = \lambda_0 \delta_{t_0} $ is a continuous mapping from $C_r(L_1)$ to $\mathbb{R}$, which is also impossible.
\item $T$ surjective implies $\varphi (Z_1\cap \mathcal{O}_2) \subseteq \mathcal{O}_1$. Suppose, on the contrary that there exists $s_0\in Z_1\cap \mathcal{O}_2$ with $\varphi (s_0)\in F_1$. By $(\ref{eq 2 thm 32})$, $$T(f) (s_0) = T(1)(s_0) \Re\hbox{e} f(\varphi(s_0)),$$ for every $f\in C_r(L_1)$. It follows from the surjectivity of $T$, together with the condition $s_0\in \mathcal{O}_2,$ that for every complex number $\omega$ there exists a real $\lambda$ satisfying $\omega = T(1)(s_0) \lambda,$ which is impossible.
\item Suppose $T$ is surjective and fix $s_0\in Z_1\cap \mathcal{O}_2$. The mapping $\delta_{s_0} T$ is a bounded real-linear mapping from $C_r(L_1)$ onto $\mathbb{C}$. On the other hand, by $(\ref{eq 2 thm 32})$, $$\delta_{s_0} T(f) = T(1)(s_0) \Re\hbox{e} f(\varphi(s_0)) + T(i) (s_0) \Im\hbox{m} f(\varphi(s_0)), \ (\forall f\in C_r(L_1)).$$ Thus, $T$ being surjective implies that the space $\mathbb{C}_{\mathbb{R}} = \mathbb{R} \times \mathbb{R}$ is linearly spanned by the elements $T(1)(s_0)$ and $T(i)(s_0)$.  Therefore, for each $s_0\in Z_1\cap \mathcal{O}_2,$ the set $\{T(1)(s_0), T(i)(s_0)\}$ is a basis of $\mathbb{C}_{\mathbb{R}} = \mathbb{R} \times \mathbb{R}.$ Consequently, when $T$ is surjective and $s_0\in Z_1\cap \mathcal{O}_2,$ the condition $T(f) (s_0)= 0$ implies $f(\varphi(s_0))=0.$ For any other $s_1\in Z_1\cap \mathcal{O}_2$ with $\varphi (s_0) = \varphi (s_1)$, we have: $$T(f)(s_0)=0 \Rightarrow f(\varphi(s_0))=0 \Rightarrow T(f)(s_1)=0.$$ The fact that $C_r(L_2)$ separates points implies that $s_1 = s_0$. Thus, $\varphi$ is injective on $Z_1\cap \mathcal{O}_2$.
\end{enumerate}}
\end{remark}

 We can now state the main result of this section which affirms that every orthogonality preserving
linear bijection between unital commutative real C$^*$-algebras is (automatically) continuous.

\begin{theorem}\label{t aut cont} Every orthogonality preserving linear bijection between unital commutative (real)
C$^*$-algebras is (automatically) continuous.
\end{theorem}

\begin{proof}
Since $T$ is surjective, $Z_3=\emptyset$, and hence $Z_1= L_2\backslash Z_2$ is a compact subset of $L_2.$ It is also clear that $\varphi (L_2)$ is compact. We claim that $\varphi (L_2) = L_1.$ Otherwise, there would exist a non-zero function $f\in C_r(L_1)$ with $\overline{\hbox{coz}(f)}\subseteq L_1\backslash \varphi (L_2)$. Thus, by $(\ref{eq varphi notin cozero})$, $T(f) =0,$ contradicting the injectivity of $T$. By Remark \ref{r consequences}$(c)$, $\varphi (Z_1)=\overline{\varphi (Z_1)}= \varphi(L_2) = \varphi (Z_1)\cup \varphi (Z_2) = L_1.$

We next see that $Z_2=\emptyset.$ Otherwise we can take $g\in C_r(L_2)$ with $\emptyset \neq coz (g)\subset Z_2.$ Let $h=T^{-1}(g).$ Obviously $Th(s)=0$ whenever $s\in Z_1.$ We claim that $h(t)=0,$ for every $t\in \varphi(Z_1)\setminus \varphi(Z_2).$ Let us fix $t\in \varphi(Z_1)\setminus \varphi(Z_2).$ Since $\varphi(Z_2)$ is a finite set there are disjoint open sets $U_1,U_2$ such that $t\in U_1,\varphi(Z_2)\subset U_2.$ Let $f\in C(L_1,\RR)$ be such that $f(t)\neq 0$ and $\overline{coz(f)}\subset U_1.$ We see that $T(fh)=0.$ Indeed, let $s\in L_2=Z_1 \cup Z_2.$ If $s$ lies in $Z_1,$ then the maps $fh$ and $f(\varphi(s))h$ lie in $C_r(L_1)$ and coincide at $\varphi(s).$ Since $T$ is linear over $\RR$ and $f$ takes real values, we deduce, by $(\ref{eq 2 thm 32})$, that   $T(fh)(s)=f(\varphi(s))Th(s)=0.$ If $s\in Z_2$ then, since $\varphi(s) \notin \overline{coz(fh)}$, then $\delta_s T(fh)=T(fh)(s)=0$.

We have shown that $T(fh)=0$. Thus, since $T$ is injective, $fh=0$ and therefore $h(t)=0.$ We have therefore proved that
$coz(h)\subset \varphi(Z_2)$ which is a finite set. This means that $h$ must be a finite linear combination of characteristic function on points of $\varphi(Z_2)$ and these points must be isolated which is impossible, since by $c)$ in Remark \ref{r consequences} no point in $\varphi(Z_2)$ can be isolated. We have proved that $Z_2=\emptyset.$ Now the fact that $T$ is continuous follows easily.
\end{proof}

The above theorem is the first step toward extending, to the real
setting, those results proved in \cite{Jar}, \cite{ArauJar}, \cite{BurGarPe}, \cite{OikPeRa}, \cite{LeTsaWon} and \cite{Tsa} for (complex) C$^*$-algebras.\smallskip

Orthogonality preserving linear bijections enjoy an interesting additional property.

\begin{proposition}\label{p inverse preserves invertible} In the notation of this section, let $T: C_r(L_1)\to C_r(L_2)$ be an orthogonality preserving linear bijection. Then $T^{-1}$ preserves invertible elements, that is, $T^{-1} (g)$ is invertible whenever $g$ is an invertible element in $C_r(L_2)$.
\end{proposition}

\begin{proof}
 Take an invertible element $g\in C_r(L_2)$. Let $f$ be the unique element in $C_r(L_1)$ satisfying $T(f) = g$. Theorem \ref{t op between Cr(K)} implies that $$0\neq g(s) =  T(f) (s)=  T(1)(s) \ \Re\hbox{e}f({\varphi (s)}) +  T(i) (s) \ \Im\hbox{m}f({\varphi (s)}),$$ for every $s\in Z_1$. This assures that $f({\varphi (s)})\neq 0,$ for every $s\in Z_1,$ and since $\varphi (Z_1) =L_1,$ $f=T^{-1} (g)$ must be invertible in $C_r(L_1)$.
\end{proof}

In the setting of complex Banach algebras, it follows from the Gleason-Kahane-\.{Z}elazko
theorem that a linear transformation $\phi$ from a unital, commutative, complex Banach algebra $A$ into $\mathbb{C}$ satisfying $\phi(1) =1$ and $\phi(a) \neq 0$ for every invertible element $a$ in $A$ is multiplicative, that is, $\phi (a b) = \phi (a) \phi(b)$ (see \cite{Glea,KaZe}). Although, the Gleason-Kahane-\.{Z}elazko theorem fails for real Banach algebras, S.H. Kulkarni found in \cite{Kul} the following reformulation: a linear map $\phi$ from a real unital Banach algebra ${A}$ into the complex numbers is multiplicative if $\varphi(1) =1$ and $\phi(a)^2+\phi(b)^2\neq 0$ for every $a, b\in{A}$ with $ab=ba$ and $a^2+b^2$ invertible. It is not clear that statement $(b)$ in the above proposition can be improved to get the hypothesis of Kulkarni's theorem. The structure of orthogonality preserving linear mappings between $C_r(L)$-spaces described in Theorem \ref{t op between Cr(K)} invites us to affirm that they are not necessarily multiplicative.\smallskip

\subsection{Bi-orthogonality preservers}

 As a consequence of the description of orthogonality preserving linear maps given in \cite{Jar}, it can be shown that an orthogonality preserving linear bijection between (complex) $C(K)$-spaces is
 bi-orthogonality preserving. It is natural to ask wether every orthogonality preserving linear bijection between commutative (unital) real C$^*$-algebras is bi-orthogonality preserving.

This is known to be true in two cases: first, between spaces $C_{\mathbb{R}}(K)$ of real (and also complex) valued functions on a compact Hausdorff space $K$, as it is well-known; second, between spaces of the type $C_{\mathbb{R}}(K;\mathbb{R}^n)$ (compare \cite[Section 3]{Dubarbie10}). Spaces like those we are dealing with in this paper need not satisfy this property, that is, there exists an orthogonality preserving linear bijection $T: C_r(L_1)\to C_r(L_2)$ which is not bi-orthogonality preserving (and even
$L_1$ and $L_2$ are not homeomorphic either).

\begin{example}\label{example OP+bijection - BiOP} Let $L_1 =
\{ t_1, t_2, t_3\}$ $L_2 = \{ s_1, s_2, s_3, s_4\}$ with $\mathcal{O}_1 = \{t_1, t_3\}$, $\mathcal{O}_2 = \{s_1\}$, $F_1 = \{ t_2\}$ and $F_2 = \{ s_2, s_3, s_4\}$. Define $\varphi : L_2 \to L_1$ by $\varphi (s_i) = t_i,$ for $i = 1, 2,$ and $\varphi (s_i) = t_3,$ for $i = 3, 4$. It is easy to check that $T(f)(s_i) =
f(\varphi(s_i))$ if $i = 1, 2$, and $T(f)(s_3) = \Re\hbox{e}f(t_3),$ $T(f)(s_4) = \Im\hbox{m}f(t_3)$ is an orthogonality preserving linear bijection, but $T^{-1}$ is not orthogonality preserving.
\end{example}

In the above example, $\varphi^{-1} (\mathcal{O}_1) \cap F_2 $ is non-empty.
Our next result shows that a topological condition on $F_2$
assures that an orthogonality preserving linear bijection between unital commutative real C$^*$-algebras
is bi-orthogonality preserving.

\begin{proposition}\label{p topological conditions}
In the notation of this section, let $T: C_r(L_1)\to C_r(L_2)$ be
an orthogonality preserving linear bijection (not assumed to be
bounded). The following statements hold:
\begin{enumerate}[$(a)$]
\item If $T$ is bi-orthogonality preserving then $\varphi : L_2 \to L_1$ is a (surjective) homeomorphism, $\varphi(F_2) = F_1,$ and $\varphi (\mathcal{O}_2) = \mathcal{O}_1$. In particular, $\varphi^{-1} (\mathcal{O}_1) \cap F_2=
 \emptyset$.
\item If $F_2$ has empty interior  then $T$ is biorthogonality preserving.\end{enumerate}
\end{proposition}

\begin{proof} $(a)$ If $T$ is bi-orthogonality preserving, it can be easily seen that $\varphi: L_2 \to L_1$ is a homeomorphism, and for each $s\in L_2$, supp$(\delta_{\varphi(s)} T^{-1}) = \{ s\}$. By Remark \ref{r consequences}$(d)$, applied to $T$ and $T^{-1}$, we have $\varphi(F_2) = F_1$ and $\varphi (\mathcal{O}_2) = \mathcal{O}_1$. Then
$\varphi^{-1} (\mathcal{O}_1) \cap F_2= \emptyset.$ So, $a)$ is clear.

$ (b).$ Let us assume that $F_2$ has empty interior. Arguing by
contradiction we suppose that $T^{-1}$ is not orthogonality
preserving. Then there exist $f_1, f_2 \in C_r(L_1)$ with $f_1 f_2
\neq 0$, but $T(f_1) \perp T(f_2)$. Thus $U:= \hbox{coz}(f_1) \cap
\hbox{coz}(f_2)$ is a non-empty open subset of $L_1.$ Keeping
again the notation of Theorem \ref{t op between Cr(K)} for $T$, we
recall that, by Theorem \ref{t aut cont} 
and Remark \ref{r consequences}, $Z_3= \emptyset,$
$Z_2=\emptyset$, $\varphi(L_2) = L_1$, $\varphi (\mathcal{O}_2)
\subset \mathcal{O}_1,$ $\varphi|_{\mathcal{O}_2}$ is injective, and for each $s\in \mathcal{O}_2$, and $\{T(1)
(s), T(i) (s)\}$ is a basis of $\mathbb{C}_{\mathbb{R}} =
\mathbb{R}\times \mathbb{R}.$

By the form of $T$, there are no points of $\varphi (\mathcal{O}_2)$ in $U = \hbox{coz}(f_1) \cap \hbox{coz}(f_2)$ (because for each $s \in \mathcal{O}_2,$ $T(f)(s) \neq 0$ when $f(\varphi(s)) \neq 0$). Now, let $k$ be a non-zero element in $C(L_1,\mathbb{R})$, with coz$(k) \subseteq \hbox{coz}(f_1) \cap \hbox{coz}(f_2)$. By Theorem \ref{t op between Cr(K)} $(\ref{eq 2 thm 32})$, it is clear that $\varphi(\hbox{coz}(T(k))) \subseteq \hbox{coz} (k)$, and hence, since  $\varphi (\mathcal{O}_2) \subseteq \mathcal{O}_1$, $\hbox{coz}(T(k))$ is a non-empty subset of $F_2$, against our hypotheses.
\end{proof}

As we have already seen, an orthogonality preserving linear bijection between $C_r(L)$-spaces needs not to be
biorthogonality preserving. Example \ref{example OP+bijection - BiOP} also shows that, unlike in the complex case,
the existence of an orthogonality preserving linear bijection between $C_r(L)$-spaces does not guarantee that the corresponding compacts spaces are homeomorphic. We next provide a characterisation of those (linear) mappings which are bi-orthogonality preserving. As a consequence, we shall see that if there exists a bi-orthogonality preserving linear map $T:C_r(L_1)\to C_r(L_2)$ then $L_1$ and $L_2$ are homeomorphic.\smallskip

 \begin{theorem}\label{t charcactersitation biop} Let $T:
 C_r(L_1)\to C_r(L_2)$ be a mapping. The following statements are
 equivalent:

\begin{enumerate}[$(a)$] \item $T$ is a bi-orthogonality preserving linear surjection;
\item There exists a (surjective) homeomorphism $\varphi:L_2\to L_1$ with $\varphi(\mathcal{O}_2)=\mathcal{O}_1,$ a function $a_1=\gamma_1+i \gamma_2$ in $C_r(L_2)$ with $a_1(s)\neq 0$ for all $s\in L_2,$ and a function $a_2=\eta_1+i \eta_2:L_2\to\CC$ continuous on $\mathcal{O}_2$ with the property that  $$ 0<\inf_{s\in \mathcal{O}_2}\left| \det\left( \begin{array}{cc} \gamma_1(s) & \eta_1(s) \\ \gamma_2(s) & \eta_2(s) \\ \end{array} \right) \right|\leq \sup_{s\in \mathcal{O}_2}\left|\det\left( \begin{array}{cc} \gamma_1(s) & \eta_1(s) \\ \gamma_2(s) & \eta_2(s) \\ \end{array} \right) \right|<+\infty ,$$ such that $$ T(f)(s)=a_1(s) \ \Re\hbox{e}f({\varphi (s)}) +  a_2(s) \ \Im\hbox{m}f({\varphi (s)})$$ for all $s\in L_2$ and $f\in C_r(L_1).$
\end{enumerate}
\end{theorem}

\begin{proof} $(a)\Rightarrow (b)$. Since every bi-orthogonality preserving linear mapping is injective, we can assume that $T:C_r(L_1)\to C_r(L_2)$ is a bi-orthogonality preserving linear bijection. We keep the notation given in Theorem \ref{t op between Cr(K)}. We have already shown that $Z_3= \emptyset,$ $Z_2=\emptyset$, $\varphi: L_2\to L_1$ is a surjective homeomorphism, $\varphi (\mathcal{O}_2) = \mathcal{O}_1,$ and for each $s\in \mathcal{O}_2$, $\{T(1) (s), T(i) (s)\}$ is a basis of $\mathbb{C}_{\mathbb{R}} =
\mathbb{R}\times \mathbb{R}$ (compare Theorem \ref{t aut cont}, Remark \ref{r consequences} and Proposition \ref{p topological conditions}). Taking $a_1 = T(1)= \gamma_1 +i \gamma_2 $ and $a_2 = T(i)= \eta_1 +i\eta_2$ we only have to show that $$ 0<\inf_{s\in \mathcal{O}_2}\left| \det\left( \begin{array}{cc} \gamma_1(s) & \eta_1(s) \\ \gamma_2(s) & \eta_2(s) \\ \end{array} \right) \right|\leq \sup_{s\in \mathcal{O}_2}\left|\det\left( \begin{array}{cc} \gamma_1(s) & \eta_1(s) \\ \gamma_2(s) & \eta_2(s) \\ \end{array} \right) \right|<+\infty .$$ Let us denote $M_s= \left(
      \begin{array}{cc}
      \gamma_1(s) & \eta_1(s) \\
      i \gamma_2(s) & i \eta_2(s) \\
      \end{array}
      \right).$ Clearly $\det (M_s) \neq 0$, for every $s\in \mathcal{O}_2$ and $T(f) (s) = M_s \cdot \left(
      \begin{array}{c}
      \Re\hbox{e} f(\varphi(s))  \\
      \Im\hbox{m} f(\varphi(s))  \\
      \end{array}
      \right),$ for every $f\in C_r(L_1), s\in L_2$. By the boundedness of $T(1):L_2 \to \mathbb{C}$ and $T(i)|_{\mathcal{O}_2} : \mathcal{O}_2 \to \mathbb{C}$ (see Lemma \ref{l 3.1}) there exists $M>0$ such that $\left|det(M_s)\right|\leq M$ for all $s\in \mathcal{O}_2.$

Applying the above arguments to the mapping $T^{-1}$ we find a surjective homeomorphism $\psi = \varphi^{-1}: L_1\to L_2$, a mapping $T^{-1} (i): L_1 \to L_2$ and $m>0,$ such that $\psi (\mathcal{O}_1) = \mathcal{O}_2$, for each $t\in \mathcal{O}_1$, $\{T^{-1}(1) (t), T^{-1}(i) (t)\}$ is a basis of $\mathbb{C}_{\mathbb{R}} =
\mathbb{R}\times \mathbb{R},$ $T^{-1} (g) (t) = N_t \cdot \left(
      \begin{array}{c}
      \Re\hbox{e} g(\psi(t))  \\
      \Im\hbox{m} g(\psi(t))  \\
      \end{array}
      \right)$ $(g\in C_r(L_2), t\in L_1)$, $\left|\det (N_t) \right|\leq m,$ for all $t\in \mathcal{O}_1,$ where $N_t = \left(
      \begin{array}{cc}
      \Re\hbox{e} T^{-1}(1)(t) & \Re\hbox{e} T^{-1}(i)(t) \\
      i \Im\hbox{m} T^{-1}(1)(t) & i \Im\hbox{m} T^{-1}(i)(t) \\
      \end{array}
      \right).$ It can be easily seen that, for each $s\in \mathcal{O}_2,$ $N_{\varphi(s)} = M_{s}^{-1},$ which shows that $\left|det(M_s)\right|\geq \frac{1}{m}$, for all $s\in \mathcal{O}_2.$

$(b)\Rightarrow (a)$. Let $T:C_r(L_1):\to C_r(L_2)$ be a mapping satisfying the hypothesis in $(b).$ Clearly, $T$ is linear, and since $\varphi (F_2) = F_1,$ $Tf(s)\in \RR$ for all $s\in F_2$ and $f\in C_r (L_1)$ (that is, $T(f) \in C_r (L_2)$). We can easily check that, under these hypothesis, $T$ is injective and preserves orthogonality.

We shall now prove that $T$ is surjective. Indeed, for each $s\in \mathcal{O}_2$
$$T(f)(s)=\left( \begin{array}{c}
      \Re\hbox{e}g(s) \\
      \Im\hbox{m}g(s) \\
      \end{array}
      \right)=\left(
      \begin{array}{cc}
      \gamma_1(s) & \eta_1(s) \\
      i \gamma_2(s) & i \eta_2(s) \\
      \end{array}
      \right) \cdot \left( \begin{array}{c}
      \Re\hbox{e}f({\varphi (s)}) \\
      \Im\hbox{m}f({\varphi (s)}) \\
      \end{array}
      \right)$$ $$= M_s \cdot \left( \begin{array}{c}
      \Re\hbox{e}f({\varphi (s)}) \\
      \Im\hbox{m}f({\varphi (s)}) \\
      \end{array}
      \right),$$ thus,
      $$\left( \begin{array}{c}
      \Re\hbox{e}f({\varphi (s)}) \\
      \Im\hbox{m}f({\varphi(s)}) \\
      \end{array}
      \right)=M_s^{-1} \cdot \left( \begin{array}{c}
      \Re\hbox{e}g(s) \\
      \Im\hbox{m}g(s) \\
      \end{array}
      \right).$$

We define $b_1(t):L_1\to \CC$ and $b_2:\mathcal{O}_1\to\CC$ by $b_1 (t) =\widetilde{ \gamma}_1 (t) + i \widetilde{ \gamma}_2  (t)$ and $b_2 =\widetilde{ \eta}_1  (t)+ i \widetilde{ \eta}_2 (t)$ ($t\in \mathcal{O}_1$), where $M_{\varphi^{-1}(t)}^{-1} = \left(
\begin{array}{cc}
\widetilde{\gamma}_1(t) & \widetilde{\eta}_1(t) \\
i \widetilde{\gamma}_2(t) & i \widetilde{\eta}_2(t) \\
\end{array}
\right),$ and $b_1(t)=\frac{1}{\gamma_1(\varphi^{-1}(t))},$ for every $t\in F_1$.
Then $S:C_r(L_2)\to C_r(L_1),$ defined by $ S(g)(t)= b_1(t)\ \Re\hbox{e}g(\varphi^{-1}(t)) + b_2(t)\ \Im\hbox{m}g(\varphi^{-1}(t)),$ is linear, preserves orthogonality and it is easy to check that $S=T^{-1}.$
It follows that $T$ is bi-orthogonality preserving.
\end{proof}

Let $T$ be a bi-orthogonality preserving linear mapping with associated homeomorphism $\varphi: L_2 \to L_1$. Clearly, the operator $S: C_r(L_1)\to C_r(L_2)$, $S(f) (s) := f(\varphi(s))$ is a $^*$-isomorphism. Having in mind that a linear mapping $T: A\to B$ between real C$^*$-algebras is a $^*$-isomorphism if, and only if,
the complex linear extension $\widetilde{T} : A\oplus i A \to B\oplus i B$, $\widetilde{T} (a+ib) = T(a) + i T(b)$ is a $^*$-isomorphism, we get the following corollary.

\begin{corollary}\label{c surjective bio isomorphism} The following statements are equivalent: \begin{enumerate}[$(a)$]\item There exists a bi-orthogonality preserving linear bijection $T: C_r(L_1)\to C_r(L_2)$;
\item There exists a C$^*$-isomorphism $S: C_r(L_1)\to C_r(L_2)$;
\item There exists a C$^*$-isomorphism $\widetilde{S}: C(L_1)\to C(L_2)$;
\item $L_1$ and $L_2$ are homeomorphic.$\hfill\Box$
\end{enumerate}
\end{corollary}

\textbf{Acknowledgements:} The authors gratefully thank to the Referee for the constructive comments and detailed recommendations which definitely helped to improve the readability and quality of the paper.

\end{document}